\documentclass[11pt]{amsart}
\usepackage{amssymb}
\usepackage[T1]{fontenc}
\usepackage[utf8]{inputenc}
\usepackage{color}
\usepackage{appendix}
\usepackage{comment}
\usepackage{wasysym}
\usepackage{xcolor}

\definecolor{darkgreen}{rgb}{0,0.5,0}
\definecolor{darkblue}{rgb}{0,0,0.8}
\definecolor{darkred}{rgb}{0.8,0,0}
\definecolor{firebrick4}{rgb}{0.84,0.1,0.35}

\usepackage[pdfencoding=auto,colorlinks,citecolor=darkgreen,linkcolor=darkblue,urlcolor=darkred]{hyperref}

\usepackage{anysize}
\marginsize{2.5cm}{2.5cm}{2.5cm}{2.5cm}
\usepackage{soul}
\usepackage{url}

\newcommand{\DD}{\mathcal{D}}
\newcommand{\N}{\mathbb{N}}

\newcommand{\Mod}[1]{\ (\mathrm{mod}\ #1)}

\renewcommand{\leq}{\leqslant}
\renewcommand{\geq}{\geqslant}
\renewcommand{\P}{\mathbb{P}}

\newcommand{\zzg}{\zeta_{\mathcal{G}}}
\newcommand{\G}{\mathcal{G}}
\newcommand{\zzgdva}{\zeta_{\G_2}}
\newcommand{\zzgminusdva}{\zeta_{\G_{-2}}}

\theoremstyle{plain}
\newtheorem{theorem}{Theorem}
\newtheorem{proposition}[theorem]{Proposition}
\newtheorem{lemma}[theorem]{Lemma}
\newtheorem{corollary}[theorem]{Corollary}

\theoremstyle{remark}
\newtheorem{remark}[theorem]{\textit{\textbf{Remark}}}
\newtheorem*{definition*}{Definition}
\newtheorem{definition}{Definition}

\title{Asymptotics of $D(q)$-pairs and triples via $L$-functions of Dirichlet characters 
}

\author[N. Adžaga]{Nikola Adžaga}
\address{Department of Mathematics, Faculty of Civil Engineering, University of Zagreb, Croatia}
\email{nikola.adzaga@grad.unizg.hr}
\author[G. Dražić]{Goran Dra\v zi\' c}
\address{Faculty of Food Technology and Biotechnology, University of Zagreb, Croatia}
\email{goran.drazic@pbf.unizg.hr}
\author[A. Dujella]{Andrej Dujella}
\address{Department of Mathematics, Faculty of Science, University of Zagreb, Croatia}
\email{duje@math.hr}
\author[A. Pethő]{Attila Pethő}
\address{Department of Computer Science, University of Debrecen, Hungary}
\email{petho.attila@unideb.hu}


\begin{document}

\begin{abstract}

Let $q$ be a non-zero integer. A \emph{$D(q)$-$m$-tuple} is a set of $m$ distinct positive integers $\{a_1, a_2, \dots, a_m\}$ such that $a_ia_j+q$ is a perfect square for all $1 \leqslant i < j \leqslant m$. By counting integer solutions $x\in[1,b]$ of congruences $x^2 \equiv q \Mod{b}$ with $b\leqslant N$, we count $D(q)$-pairs 
with both elements up to $N,$ and give estimates on asymptotic behaviour. We show that for prime $q$, the number of such $D(q)$-pairs and $D(q)$-triples grows linearly with $N$. Up to a factor of $2$, the slope of this linear function is the quotient of the value of the $L$-function of an appropriate Dirichlet character (usually a Kronecker symbol) and of $\zeta(2)$.

\end{abstract}

\maketitle

\section{Introduction}

A Diophantine pair is a set of two positive integers $\{a, b\}$ such that $ab+1$ is a perfect square. Dujella has proven that the number of Diophantine pairs with both elements less than or equal to $N$ asymptotically grows as $\frac{6}{\pi^2}N\log N$ in \cite{DujeGlavni} (while the error term was further estimated in \cite{Lao}). The problem is equivalent to counting solutions of $x^2 \equiv 1 \Mod{n},$ where  $n$ ranges from $1$ to $N$. This congruence has $2^{\omega(n)}$ solutions for each $n$ (where $\omega(n)$ denotes the number of distinct prime factors of $n$), so the problem is reduced to estimating the sum $\displaystyle\sum_{n=1}^N 2^{\omega(n)}$.

In \cite{DujeGlavni}, it was also shown that the number of Diophantine triples (i.e.~the number of sets of three positive integers such that the product of any two is one less than a perfect square) is roughly half of the number of Diophantine pairs, while the number of Diophantine quadruples was shown to have the order of magnitude of $\sqrt[3]{N}\log N$.  Martin and Sitar in \cite{MartinSitar} have then determined that the number of Diophantine quadruples with all elements less than or equal to $N$ asymptotically grows as $\dfrac{2^{4/3}}{3\Gamma(2/3)^3}\, \sqrt[3]{N}\log N$.

One natural generalization of Diophantine $m$-tuples is obtained by replacing $1$ with a different but fixed non-zero integer: \emph{$D(q)$-$m$-tuple} is a set of $m$ positive integers $\{a_1, \dotsc, a_m\}$ such that $a_ia_j+q$ is a perfect square for all $1\leqslant i < j \leqslant m$. So far, infinitely many $D(q)$-quadruples have been found only for square numbers $q$. Therefore, we wish to estimate the number of $D(q)$-pairs and $D(q)$-triples. Denote by $D_{m,q}(N):=| \lbrace S\subset \lbrace 1, 2, \dotsc, N\rbrace \colon S$ is a $D(q)$-$m$-tuple \hspace{-0.6em} $\rbrace|.$ 

We often deal with quadratic congruences $x^2 \equiv q \pmod{b}$, where $b$ and $q$ are integers and $b\geq 2$. Under the number of its solutions we mean the number of integers $x\in [1,b]$ satisfying it.

Let $q$ be any integer such that $|q|$ is prime. In this paper we estimate $D_{2,q}(N)$, the number of $D(q)$-pairs $(a,b)$ such that $a<b$ where $b$ ranges from $1$ to $N$. We do this by counting the number of solutions of congruences 
\begin{equation} 
 x^2 \equiv q \Mod b
\label{eq:01}\end{equation}  where $b$ ranges from $1$ to $N$. We will prove asymptotic estimates on the number of solutions to said congruence equations, and these estimates will easily translate to $D_{2,q}(N),$ since the two quantities differ by $O(1).$ 

In Section \ref{sec:equi} we explain how the problem reduces to counting solutions of congruences \eqref{eq:01}. We use quadratic reciprocity to characterize the moduli $b$ such that equation \eqref{eq:01} has a solution, and express the number of its solutions (in a complete residue system) as a function of the number of distinct prime factors of $b$. Finally, we proceed to estimate the relevant weighted sums (of $2^{\omega(n)}$) by analyzing their Dirichlet series and applying a tauberian theorem.

$L$-functions of Dirichlet characters appear in our results -- we use the same notation as in LMFDB \cite{lmfdb}, and the relevant background is stated in the Appendix of this paper to make it self-contained.

Here we state our results for prime $2$ (and $-2$). 

\begin{theorem}\label{tm:pm2}
The number of $D(2)$-pairs with both elements in the set $\lbrace1, 2, \dots ,N\rbrace$ satisfies
\[ D_{2,2}(N) \sim \frac{L(1,\chi_{8,5})}{\zeta(2)}\cdot N \approx 0.37888N, \] 
whereas the number of $D(-2)$-pairs with both elements in the set $\lbrace1, 2, \dots ,N\rbrace$ satisfies
\[ D_{2,-2}(N) \sim \frac{L(1,\chi_{8,3})}{\zeta(2)}\cdot N \approx 0.67524N. \]
\end{theorem}

The estimates for other primes $q$ are more involved. The results depend on the remainder of $q$ modulo $8$ (i.e.\ on the power of $2$ dividing $q-1$), and it turns out that the relevant Dirichlet character is always the Kronecker symbol.

\begin{theorem}\label{tm:sviq} Let $q$ be an integer such that $|q|$ is a prime or $q=-1$, and denote by $D_{2,q}(N)$ the number of $D(q)$-pairs with both elements in the set $\{1,2,\ldots,N\}$. 

\begin{itemize}
    \item [a)]
If $q\equiv 3 \Mod{4}$, then
\[ D_{2,q}(N)  \sim \frac{L(1,\chi_{4|q|,4|q|-1})}{\zeta(2)}\cdot N . \]
    \item [b)] If $q\equiv 5 \Mod{8}$, then
\[ D_{2,q}(N)  \sim \frac{2L(1,\chi_{|q|,|q|-1})}{\zeta(2)}\cdot N . \]
\item [c)] If $q\equiv 1 \Mod{8}$, then
\[ D_{2,q}(N)  \sim \frac{L(1,\chi_{|q|,|q|-1})}{\zeta(2)}\cdot N . \]
\end{itemize}
\end{theorem}

In the last section, for any integer $n$, we relate the number of $D(n)$-triples with all elements up to $N$ to the number of $D(n)$-pairs. More precisely, we show the following theorem.

\begin{theorem}\label{tm:trojkeUvod}
Let $n$ be a non-zero integer. The number of $D(n)$-triples with all elements in the set $\{1, 2, \dotsc, N\}$ is asymptotically equal to half the number of $D(n)$-pairs. More precisely, 
\[ D_{3,n} (N) \sim \frac{D_{2,n}(N)}{2}. \]
\end{theorem}

Finally, we list the estimates we obtain (by applying Theorem \ref{tm:pm2} and Theorem \ref{tm:sviq}) on the number of $D(q)$-triples for integers $q$ such that $|q|$ is prime.

We note here that there is a conjecture stating that there are only finitely many $D(n)$-quadruples when $n$ is not a perfect square \cite[Conjecture 1.5.2]{DujeNovaKnjiga} and this conjecture is confirmed in \cite{BCM} for $q=-1, -4$ and in \cite{Brown} for $q\equiv 2 \Mod{4}$.

\section{Reducing the problem to congruences}\label{sec:equi}

The results of this paragraph hold for an arbitrary integer $q$ (not necessarily prime). We estimate the number of $D(q)$-pairs $(a,b)$ such that $a<b$ using the number of solutions $x$ of the equation \eqref{eq:01}.
Almost all such solutions induce a $D(q)$-pair $(a,b)$ such that $a\leq b,$ simply by setting $a=\frac{x^2-q}{b}.$
The almost part comes from the fact that $x^2-q$ can be negative, but the total amount of such cases for all $b\in \N$ is finite. Also, almost all $D(q)$ pairs $(a,b)$ such that $a<b$ are induced by a solution of Equation (\ref{eq:01}). It is possible that there exist pairs $(a,b), a<b$ such that $a>\frac{x^2-q}{b}$ for all solutions of Equation (\ref{eq:01}). Then there exists some $x\geq b+1$ such that $a=\frac{x^2-q}{b}.$ This leads to $b\geq a \geq \frac{(b+1)^2-q}{b}$, and in turn to $b\leq \frac{q-1}{2}.$ All in all, again only finitely many cases when $b$ runs through $\N.$ For the sake of our calculations, we identify the number of $D(q)$-pairs $(a,b)$ such that $a<b$ with the number of solutions of Equations (\ref{eq:01}).

\begin{lemma}\label{lema:prva} Let $q$ be an integer such that $|q|$ is prime and $b\in \N$ such that $\gcd(b,2q)=1.$ The number of solutions of the congruence
\begin{equation}
 x^2 \equiv 1 \Mod b
\label{eq:02}
\end{equation} such that $1\leq x \leq b$ is $2^{\omega(b)}.$ Consequently, the number of solutions of the congruence
\begin{equation}
 x^2 \equiv q \Mod b
\tag{\ref{eq:01}}
\end{equation} such that $1\leq x \leq b$ is either zero or $2^{\omega(b)}$.
\end{lemma}
\begin{proof}
The first statement of the lemma is proved in \cite[Section V.4]{Vinogradov}. If there is no solution to Equation \eqref{eq:01}, we are done. If there exists a solution $x_q,$ then every other solution $x'$ of Equation \eqref{eq:01} satisfies 
\begin{equation}
\left(\frac{x'}{x_q}\right)^2 \equiv 1 \Mod b,
\end{equation}
where division by $x_q$ corresponds to multiplying by the inverse of $x_q$ modulo $b$.
Also, if $x_1$ is any solution to Equation \eqref{eq:02}, then $x_1x_q$ is a solution of Equation \eqref{eq:01} and all solutions obtained in such a way have different residues \text{mod} $b.$
\end{proof}

We now give all the details for Theorem \ref{tm:pm2}. Estimating the number of $D(2)$-pairs is somewhat easier than estimating the number of $D(q)$-pairs for other prime $q$. However, as the proof of this theorem contains all the essential steps necessary for all other $q$, we believe that reading this first will make it easier for the reader to follow the more involved proofs.

\section{Estimating the number of \texorpdfstring{$D(2)$}{D(2)}-pairs and \texorpdfstring{$D(-2)$}{D(-2)}-pairs}

We first estimate $D_{2,2}(N)$, the number of $D(2)$-pairs up to $N$, by counting solutions of the congruence \eqref{eq:01} for $q=2$, and proceed similarly for $D_{2,-2}$.

\subsection{Existence and the number of congruence solutions}
In the next lemma, we record when the relevant congruence equation has a solution, as well as the number of its solutions.

\begin{lemma}\label{lemma:2-2}
For odd $b$, the equation
\begin{equation}\label{eq:for2}
 x^2 \equiv 2 \Mod b,
 \end{equation}
has a solution if and only if each prime factor $p$ of $b$ satisfies $p\equiv \pm 1 \Mod{8}$. For even $b$, equation \eqref{eq:for2} has a solution if and only if $2||b$ and each odd prime factor $p$ of $b$ satisfies $p\equiv \pm 1 \Mod{8}$.

Analogously, the equation
\begin{equation}\label{eq:forminus2}
x^2 \equiv -2 \Mod b 
\end{equation}
has a solution if and only if each prime factor $p$ of $b$ satisfies $p\equiv 1,3 \Mod{8}$. For even $b$, equation \eqref{eq:forminus2} has a solution if and only if $2||b$ and each odd prime factor $p$ of $b$ satisfies $p\equiv 1,3 \Mod{8}$.

When equation \eqref{eq:for2} or \eqref{eq:forminus2} is solvable with odd $b$, the number of its solutions $x$ such that $1\leqslant x \leqslant b$ is exactly $2^{\omega(b)}$.
\end{lemma}

\begin{proof} Since $b\mid x^2-2$ implies that for each prime factor $p$ of $b$ it holds that $p \mid x^2-2$, we have to check for which primes $p$ is $2$ a quadratic residue.
The statement on the existence of solutions then holds because $x^2-2$ cannot be divisible by $4$ and because \[ \left(\frac{2}{p}\right) = (-1)^{\frac{p^2-1}{8}}, \]
and analogously for $-2$. The last statement follows from Lemma \ref{lema:prva}.
\end{proof}

The previous lemma motivates us to define the set of \emph{good primes} as \[ \G_2 = \{p \in \P \colon p \equiv \pm 1 \Mod{8} \}.\]
The set of good primes for $-2$ is given by \[ \G_{-2} = \{p \in \P \colon p \equiv 1, 3 \Mod{8}\}.\]

The sum $\displaystyle \sum_{n=1}^N 2^{\omega(n)}$ is already estimated in \cite{DujeGlavni}, but now we have to estimate a weighted version of this sum. The weights are binary, i.e.\ non-zero if $n$ consists only of good prime factors:
\[
\lambda_{\mathcal{G}_2}(n)=
	\begin{cases}
		1, \quad\text{ if } n=p_1^{\alpha_1}\dots p_k^{\alpha_k}, \quad p_i\in \mathcal{G}_2, \, \forall i=1,\dotsc,k,\\
		0, \quad\text{ otherwise.}
	\end{cases}
\]

To begin estimating the weighted sum, we define \[ b_2(n)=2^{\omega(n)}\cdot \lambda_{\G_2}(n). \]
If $n$ has only good prime factors, then $b_2(n)$ is equal to the number of solutions to congruence $x^2\equiv 2 \Mod{n}$; otherwise the value of $b_2(n)$ is zero.
We want to estimate the weighted sum
\[
\displaystyle B_2(N)= \sum_{1\leq n \leq N} 2^{\omega(n)}\cdot \lambda_{\mathcal{G}_2}(n)=\sum_{1\leq n \leq N} b_2(n).
\] $B_2(N)$ counts the total number of solutions $x \in \{1, \dotsc, n\}$ of all congruences $x^2 \equiv 2 \Mod{n}$ where $n$ is odd and $1 \leq n\leq N.$ We will account for the possibility of $2||n$ later, so understanding the asymptotic behavior of $B_2(N)$ will be enough to understand $D_{2,2}(N).$

\subsection{Dirichlet series manipulation}
To understand the asymptotic behavior of $B_2(N)$, we analyze the behavior of the Dirichlet series $\beta_2(s) = \DD b_2(s)$, where
\[ \beta_2(s) = \DD b_2(s) = \sum_{n=1}^\infty \frac{b_2(n)}{n^s} = \sum_{n=1}^\infty \frac{2^{\omega(n)} \lambda_{\mathcal{G}_2}(n)}{n^s} .\] The next lemma will be used throughout the following sections as well, so we state it in a more general manner.

\begin{lemma}\label{lema:omjerzeta}
Let $\G$ be a set of primes called \emph{good primes}. Let $\lambda_G \colon \N \to \{0, 1\}$ be the indicator function of a multiplicative monoid in $\N$ generated by $\G$. Then the Dirichlet series $\beta(s)$ of $b(n)=2^{\omega(n)}\cdot \lambda_{\G}(n)$ satisfies
\[
\beta(s)=\frac{\zzg^2(s)}{\zzg(2s)},
\] for $\Re{s}>1,$
where $\zzg(s)$ is
\[ \zzg(s):=\DD \lambda_\mathcal{G}(s)=\sum_{n=1}^\infty \frac{\lambda_{\mathcal{G}}(n)}{n^s}. \]
\end{lemma}
\begin{proof}
Since Dirichlet series behave nicely with respect to Dirichlet convolution, we wish to express $b(n)$ as a convolution of two arithmetic functions. One of these functions will be \emph{the $\mathcal{G}$-modified M\" obius function} which we define as
\[
\mu_\mathcal{G}(n)=
	\begin{cases}
		(-1)^{\omega(n)}, \quad\text{ if } n=p_1\dots p_k, \quad p_i\in \mathcal{G}, \, \forall i=1,\dotsc,k, \text{ and } p_i\neq p_j \text{  whenever } i\neq j\\
		\quad \quad \quad \;0, \quad\text{ otherwise}.
	\end{cases}.
\]
Now we can express
\begin{align*}
b(n) &= 2^{\omega(n)}\cdot \lambda_\mathcal{G}(n)= \sum_{d|n} \mu^2_{\mathcal{G}}(d) \cdot \lambda_\mathcal{G}(n) \\
&\stackrel{(**)}{=} \sum_{d|n} \mu^2_{\mathcal{G}}(d) \cdot \lambda_\mathcal{G}\left(\frac{n}{d}\right) =\left(\mu^2_{\mathcal{G}} * \lambda_\mathcal{G}\right)(n)
\end{align*}

\noindent where equality $(**)$ holds because of the following fact: If $n$ only has good prime factors, then $\lambda_\mathcal{G}(\frac{n}{d})=\lambda_\mathcal{G}(n)$ for any $d$ such that $d|n.$ If $n$ has at least one bad prime factor, then $\lambda_\mathcal{G}(n)=0,$ as well as $\mu^2_{\mathcal{G}}(d)\cdot\lambda_\mathcal{G}(\frac{n}{d}) =0.$

Since $\displaystyle \DD (\mu^2_{\mathcal{G}}*\lambda_{\mathcal{G}}) (s)=\DD \mu^2_{\mathcal{G}} (s) \DD \lambda_{\mathcal{G}} (s),$ we only need to calculate $\DD \mu^2_{\mathcal{G}} (s).$
As $\mu^2_\G$ is multiplicative, we can expand $\DD \mu^2_{\mathcal{G}} (s)$ into an Euler product (see e.g.~\cite[Theorem 1.9]{MontV} -- this theorem is also stated herein at the end of the Appendix as Theorem \ref{tm:MV}) to obtain

$\displaystyle \DD(\mu^2_{\mathcal{G}})=\prod_{p\in \mathcal{G}} \left(1+\frac{1}{p^s}\right)=\frac{\prod_{p\in \mathcal{G}} \left(1-\frac{1}{p^{2s}}\right)}{\prod_{p\in \mathcal{G}} \left(1-\frac{1}{p^s}\right)}=\frac{\prod_{p\in \mathcal{G}} \left(1-\frac{1}{p^s}\right)^{-1}}{\prod_{p\in \mathcal{G}} \left(1-\frac{1}{p^{2s}}\right)^{-1}}=\frac{\zeta_{\mathcal{G}}(s)}{\zeta_{\mathcal{G}}(2s)}$ .
\end{proof}

We obtain the following corollary by noting that our $\lambda_{\G_2}$ and $\lambda_{\G_{-2}}$ are indicator functions as required by the previous lemma.

\begin{corollary}\label{kor:Zeta2Trunc}
The Dirichlet series $\beta_2(s)$ and $\beta_{-2}(s)$ satisfy
\begin{equation}
    \beta_2(s)=\frac{\zzgdva^2(s)}{\zzgdva(2s)}, \quad \beta_{-2}(s)=\frac{\zzgminusdva^2(s)}{\zzgminusdva(2s)},
\end{equation}
where $\zzgdva$ and $\zzgminusdva$ are the Dirichlet series of $\lambda_{\G_2}$ and $\lambda_{\G_{-2}}$.
\end{corollary}

Using the previous corollary, we show how $\beta_2$ and $\beta_{-2}$ can be expressed in terms of the classical zeta function and the $L$-functions of certain Dirichlet characters.

\begin{lemma}\label{lema:zeteZa2} The following holds.
\begin{itemize}
\item[a)] The Dirichlet series $\beta_2(s)=\DD b_2(s)$ of $b_2(n)=2^{\omega(n)}\cdot \lambda_{\G_2}(n)$ satisfies
\[\beta_2(s) = \frac{\zeta(s)}{\zeta(2s)} \cdot \frac{L(s,\chi_{8,5})}{(1+2^{-s})}.  \]
\item[b)] The Dirichlet series $\beta_{-2}(s) = \DD b_{-2}(s)$ of $b_{-2}(n)=2^{\omega(n)}\cdot \lambda_{\G_{-2}}(n)$ satisfies
\[\beta_{-2}(s) = \frac{\zeta(s)}{\zeta(2s)} \cdot  \frac{L(s,\chi_{8,3})}{(1+2^{-s})}.  \]
\end{itemize}
\end{lemma}

\begin{proof}
We begin proving a) by complementing the Euler product of $\zzgdva$ from Corollary \ref{kor:Zeta2Trunc} to obtain the usual zeta function:
\begin{align*}
    \zzgdva(s) &= \prod_{p\in\G_2} (1-p^{-s})^{-1} = \zeta(s) \prod_{p\notin\G_2}(1-p^{-s}) \\
    &= \zeta(s)(1-2^{-s})\prod_{p\equiv 3}(1-p^{-s})\prod_{p\equiv 5}(1-p^{-s}),
\end{align*}
where the products go over all primes $p$ congruent to $3$ and $5$ modulo $8$ (according to our description of $\G_2$, the primes that are not in $\G_2$ include $2$ and all primes of this form). The further products will also go over congruences modulo $8$.

We now rewrite
\begin{align*}
 \frac{\zzgdva^2(s)}{\zzgdva(2s)} &= \frac{\zeta^2(s)}{\zeta(2s)} \cdot \frac{(1-2^{-s})^2}{(1-2^{-2s})} \cdot \frac{\displaystyle{\prod_{p\equiv 3}(1-p^{-s})^2\prod_{p\equiv 5}(1-p^{-s})^2}}{\displaystyle{\prod_{p\equiv 3}(1-p^{-2s})\prod_{p\equiv 5}(1-p^{-2s})}} \\
 &= \frac{\zeta^2(s)}{\zeta(2s)} \cdot \frac{(1-2^{-s})}{(1+2^{-s})} \cdot \frac{\displaystyle{\prod_{p\equiv 3}(1-p^{-s})\prod_{p\equiv 5}(1-p^{-s})}}{\displaystyle{\prod_{p\equiv 3}(1+p^{-s}) \,\, \prod_{p\equiv 5}(1+p^{-s})}}.
 \end{align*}
 We invert our products and complement them with the remaining possible remainder of an odd prime modulo $8$:
 \begin{align*}
 \frac{\zzgdva^2(s)}{\zzgdva(2s)} &= \frac{\zeta^2(s)}{\zeta(2s)} \cdot \frac{(1-2^{-s})}{(1+2^{-s})} \cdot \frac{\displaystyle{
 \prod_{p\equiv 1}(1-p^{-s})^{-1}
 \prod_{p\equiv 3}(1+p^{-s})^{-1}
 \prod_{p\equiv 5}(1+p^{-s})^{-1}
 \prod_{p\equiv 7}(1-p^{-s})^{-1}
 }}{\displaystyle{
 \prod_{p\equiv 1}(1-p^{-s})^{-1}
 \prod_{p\equiv 3}(1-p^{-s})^{-1}
 \prod_{p\equiv 5}(1-p^{-s})^{-1}
 \prod_{p\equiv 7}(1-p^{-s})^{-1}
 }} \\
 &= \frac{\zeta^2(s)}{\zeta(2s)} \cdot \frac{(1-2^{-s})}{(1+2^{-s})} \cdot \frac{L(s,\chi_{8,5})}{L(s,\chi_{8,1})} = \frac{\zeta(s)}{\zeta(2s)} \cdot \frac{L(s,\chi_{8,5})}{(1+2^{-s})}.
 \end{align*}

For b) part about $-2$, the proof is completely analogous, and the character $\chi_{8,3}$ appears instead of $\chi_{8,5}$ due to a different set of good primes $\G_{-2}$.
\end{proof}

Our Dirichlet series $\beta_2(s)$ and $\beta_{-2}(s)$ are holomorphic in the region $\Re s > 1$ by a standard analytic argument given in the Appendix as Corollary \ref{kor:holo}. The previous lemma also shows that these series have holomorphic extensions for $\Re s \geqslant 1$, except at $s=1$, which we will use in the next subsection.

\subsection{The total number of solutions of all congruences with odd moduli}

The asymptotic behaviour of $\displaystyle B_2(N)=\sum_{1\leq n \leq N} b_2(n)$ is a direct consequence of a corollary of a theorem by Wiener and Ikehara.

\begin{theorem}[Corollary of Wiener-Ikehara \cite{WIK}]
Let  $a(n) \geq 0$. If the Dirichlet series of the form

\[ \sum_{n=1}^\infty a(n) n^{-s} \]

\noindent converges to an analytic function in the half-plane $ \Re(s) \geq 1$ with a simple pole of residue $c$ at $s=1$, then

\[ \sum_{n\leq N}a(n) \sim cN.\]
\end{theorem}

Let us remind the reader that $B_2(N)$ counts the total number of solutions $x \in \{1, \dotsc, n\}$ of all congruences $x^2 \equiv 2 \Mod{n}$ where $n$ is odd and $1 \leq n\leq N.$

\begin{proposition}\label{prop:07} The following holds.
\begin{itemize}
    \item[a)] The partial sums of $b_{2}(n)$ satisfy
\[ B_2(N) \sim \frac{2L(1,\chi_{8,5})}{3\zeta(2)}\cdot N \approx 0.25258N .\]
    \item[b)] The partial sums of $b_{-2}(n)$ satisfy
    \[ B_{-2}(N) \sim \frac{2L(1,\chi_{8,3})}{3\zeta(2)}\cdot N \approx 0.45016N.\]
\end{itemize}
\end{proposition}

\begin{proof}
The function $\displaystyle B_2(N) = \sum_{n\leq N} b_2(n)$ is the partial sum of the sequence $(b_2(n))_{n\geq 1}$ with Dirichlet series rewritten in Lemma \ref{lema:zeteZa2} as
\[\beta_2(s) = \frac{\zeta(s)}{\zeta(2s)} \cdot \frac{L(s,\chi_{8,5})}{(1+2^{-s})}.  \]
The function $\beta_2(s)$ is analytic on the half-plane given by $\Re s \geqslant 1$ except for $s=1$, and to apply the previous theorem, we need the residue at $s=1$. Among all factors, only  $\zeta(s)$ is not holomorphic at $s=1$. Factors in the denominators have no zeroes for $\Re s > \frac 12$. Since $\zeta(s)$ has a simple pole at $s=1$, we will multiply its residue, which is equal to $1$, with the value of the remaining factors at $s=1$. Therefore the residue of $\beta_2(s)$ at $s=1$ is $\frac{1}{\zeta(2)}\cdot \frac{2L(1,\chi_{8,5})}{3}$ and the claim now follows by the Wiener-Ikehara theorem. Part b) is completely analogous.
\end{proof}
\begin{remark}
One could likely use Perron's formula to find the explicit error term, but this would be computationally harder than our determination of the main term.
\end{remark}

\subsection{The asymptotics of \texorpdfstring{$D_{2,2}(N)$}{D_{2,2}(N)} and \texorpdfstring{$D_{2,-2}(N)$}{D_{2,-2}(N)}}
We can now finally prove that \texorpdfstring{$D_{2,2}(N)$}{D_{2,2}(N)}, the number of $D(2)$-pairs up to $N$, grows linearly with $N$ and determine its gradient.

\begin{proof}[Proof of Theorem \ref{tm:pm2}]
The number of $D(2)$-pairs up to $N$ is equal to the number of congruence solutions $x^2 \equiv 2 \Mod{n}$ with $x\in \lbrace 1, 2, \dots, n\rbrace$ and $n \in \{1, \dotsc, N\}$ plus some $O(1)$. Let us denote the number of congruence solutions by $C_2(N)$.
We now let $n$ vary through all integers between $1$ and $N$, both odd and even. For even $n$, since $2$ and $n/2$ are coprime (due to $2||n$), the number of solutions is $2^{\omega(n/2)}=2^{\omega(n)-1}$. The total count of congruence solutions for $n \leqslant N$ is hence 
 \[   C_2(N) = \displaystyle \sum_{1\leq n \leq N}2^{\omega(n)}\cdot \lambda_\mathcal{G}(n)+\sum_{\substack{1\leq n \leq N \\ 2||n}}2^{\omega(n)-1}\cdot \lambda_\mathcal{G}\left(\frac{n}{2}\right) 
    = B_2(N) + B_2\left(\left\lfloor \frac{N}{2} \right\rfloor\right).
    \]
Since $B_2(N)\sim \frac{2L(1,\chi_{8,5})}{3\zeta(2)}\cdot N$, it follows that
$C_2(N) \sim \left(1+\frac12\right)\frac{2L(1,\chi_{8,5})}{3\zeta(2)}\cdot N = \frac{L(1,\chi_{8,5})}{\zeta(2)}\cdot N$, where the error from replacing $\lfloor N/2 \rfloor$ by $N/2$ is $O(1)$. Part b) is again completely analogous. 

\end{proof}

\section{Estimating the number of \texorpdfstring{$D(q)$}{D(q)}-pairs for odd primes \texorpdfstring{$q$}{q}}
The asymptotic estimation and its proof will have the same outline for odd primes $q$, with the following differences. In Subsection \ref{sub:exi}, we determine whether the congruence \eqref{eq:01} has a solution by using quadratic reciprocity (instead of its supplement for $\pm 2$). In Subsection \ref{sub:cnt}, we carefully analyze the number of solutions with respect to the occurrences of primes $2$ and $q$ in $n$. The usage of the Wiener-Ikehara theorem requires identifying proper characters and computing the residue in the same manner -- this is done in Subsection \ref{sub:BofN}. Expressions for $C(N)$, the total count of solutions of all congruences, are going to vary according to the possible appearances of primes $2$ and $q$ in the prime factorization of $n$. This final analysis is done in Subsection \ref{sub:fin}.

\subsection{Existence of congruence solutions}\label{sub:exi}
We first investigate when equation \eqref{eq:01} has a solution. Since the number of solutions is $0$ or $2^{\omega(b)}$, in the next several lemmas we give conditions on whether the number of solutions is non-zero, depending on the residue of $q$ modulo $8$.

\begin{lemma}\label{Lemma:2} Let $q$ be a prime with $q\equiv 3 \Mod 4.$ Equation \eqref{eq:01} has a solution if and only if $b=\displaystyle\delta\prod_{p_i \neq q} p_i^{\alpha_i}$ such that $\left(\frac{q}{p_i}\right)=1$ for all $i,$ and $\delta \in \lbrace 1, 2, q, 2q \rbrace.$ The condition $\left(\frac{q}{p_i}\right)=1$ is equivalent to $\left(\frac{p_i}{q}\right)= (-1)^{\frac{p_i-1}{2}}$.

\end{lemma}

\begin{proof}
 First we notice that no higher powers of $2$ or $q$ are possible in the factorization of $b$. The number $b$ is not divisible by $4$ since that would imply that $4$ divides $x^2-3$, whereas $b$ is not divisible by $q^2$ since then $q^2$ would divide $x^2-q$. If $\gcd(b,2q)=1$ and $x^2\equiv q \Mod b$ then exactly one of the numbers $x, x+b, x+2b, \dots, x+(2q-1)b$ will be the solution of $y^2 \equiv q \Mod{2qb}.$ This means it is enough to analyze the case $\gcd(b,2q)=1.$

We now focus on such $b.$ Assume that for a fixed $b,$ Equation \eqref{eq:01} has a solution $x_0$ and let $p|b.$ Then $x^2\equiv q \Mod p,$ which by quadratic reciprocity implies that $1=\left(\frac{q}{p}\right)=\left(\frac{p}{q}\right)\cdot (-1)^{\frac{p-1}{2}}.$ We call $p$ \emph{good for $q$} if \textbf{$\left(\frac{p}{q}\right)= (-1)^{\frac{p-1}{2}}$} .

We proved that $b$ must be of the form given in the statement of the lemma. Now we prove that \eqref{eq:01} has a solution for every such $b$.

Assume $p$ is good for $q.$ We prove by induction that $x^2 \equiv q \Mod{p^n}$ has a solution for every $n\in \N.$ The base case is true because from the fact that $p$ is good for $q$ we have that $1=\left(\frac{p}{q}\right)\cdot (-1)^{\frac{p-1}{2}}=\left(\frac{q}{p}\right)$, that is, $q$ is a quadratic residue \text{mod} $p.$ Let $x_0$ be a solution for $p^n.$ If it is also a solution for $p^{n+1},$ we are done. Otherwise, look at the numbers $x_0, x_0+p^n, x_0+2p^n,\dots, x_0+(p-1)p^n,$ more specifically, for $i\neq j,$ look at $\left[(x_0+ip^n)^2-q\right]-\left[(x_0+jp^n)^2-q\right]=(i-j)p^n(2x_0+(i+j)p^n).$ Since $\gcd(p,2q)=1$ we know that $p\nmid 2x_0+(i+j)p^n$ and trivially $p\nmid i-j,$ so the numbers $(x_0+ip^n)^2-q$ give $p$ different residues $\Mod{p^{n+1}}$ and one of these numbers must be divisible by $p^{n+1}.$

If $\gcd(b_1, b_2)=1$ and $x_i^2\equiv q \Mod{b_i},$ then $\{x_1, x_1+b_1, \dots, x_1+(b_2-1)b_1\}$ is the complete residue system mod $b_2$ so one of the elements must be a solution of the equation $x^2\equiv q \Mod{b_2}.$ As each of these numbers is also a solution to the equation $x^2\equiv q \Mod{b_1},$ then there is at least one simultaneous solution (this also follows from the Chinese remainder theorem).
\end{proof}

\begin{lemma}\label{lema:existence51}
Let $q$ be a prime with $q\equiv 5 \Mod 8.$ Equation \eqref{eq:01} has a solution if and only if $b=\delta\prod p_i^{\alpha_i}$ such that $\left(\frac{p_i}{q}\right)=1$ for all $i,$ and $\delta \in \lbrace 1, 2, 4, q, 2q, 4q \rbrace.$

Let $q$ be a prime with $q\equiv 1 \Mod 8.$ Equation \eqref{eq:01} has a solution if and only if $b=\delta\cdot 2^{\alpha_0}\prod p_i^{\alpha_i}$ such that $\left(\frac{p_i}{q}\right)=1$ for all $i,$ and $\delta \in \lbrace 1, q \rbrace.$
\end{lemma}

\begin{proof}
If $q\equiv 5 \Mod 8$, the proof mimics that of Lemma \ref{Lemma:2}. If $q\equiv 1 \Mod 8$, we only need to prove that any power of $2$ is possible as a factor of $b.$ We again do this by induction. Taking any odd $x,$ we have $x^2 \equiv q \Mod 8.$ For any odd $b$ such that there exists a solution to Equation \eqref{eq:01} we can construct a solution of $x^2\equiv q \Mod{8b}$ since $\gcd(8,b)=1,$ so a base case follows. Let $x$ be such that $x^2 \equiv q \Mod{2^{\alpha}b}$ where $b$ is odd. If it is also true that $x^2 \equiv q \Mod{2^{\alpha+1}b}$ we are done, otherwise look at $x+b\cdot 2^{\alpha-1}.$ Now we have $(x+b\cdot 2^{\alpha-1})^2-q=x^2-q+xb2^{\alpha}+b^22^{2\alpha-2}.$ The right hand side is divisible by $b$, the number $2^{2\alpha-2}$ is divisible by $2^{\alpha+1}$ since $\alpha\geq 3,$ and both numbers $x^2-q$ and $xb2^{\alpha}$ are divisible by exactly $2^{\alpha}$ so their sum is divisible by $2^{\alpha+1}.$
\end{proof}

\begin{lemma}\label{lema:predznak}
Let $p,|q|$ be odd primes such that $|q|\neq p $. Then
\[
\left(\frac{q}{p}\right)=\left(\frac{p}{|q|}\right)\cdot (-1)^{\frac{p-1}{2}\frac{q-1}{2}}
\]
\end{lemma}
\begin{proof}
If $q>0$ then this is simply quadratic reciprocity. If $q<0$ we have
\begin{align*}
\left(\frac{q}{p}\right)&=\left(\frac{-1}{p}\right)\left(\frac{|q|}{p}\right)=(-1)^{\frac{p-1}{2}}\left(\frac{p}{|q|}\right)\cdot (-1)^{\frac{p-1}{2}\frac{|q|-1}{2}}=\left(\frac{p}{|q|}\right)\cdot(-1)^{\frac{p-1}{2}\frac{|q|+1}{2}}\\
&=\left(\frac{p}{|q|}\right)\cdot(-1)^{\frac{p-1}{2}\frac{-q+1}{2}}=\left(\frac{p}{|q|}\right)\cdot (-1)^{\frac{p-1}{2}\frac{q-1}{2}}
\end{align*}
\end{proof}

\subsection{The number of congruence equation solutions}\label{sub:cnt}
The following lemma counts the number of solutions when they exist. Some results in it could be written in a shorter form, but this form was chosen to make further proofs easier to understand.

\begin{lemma}{(Extension of Lemma \ref{lema:prva})} \label{Lemma:4}
Let $|q|$ be a prime number and $b\in \N$ such that $\gcd(b,2q)=1,$ and $b$ has only good prime factors for $q.$ 

Let $q\equiv 3 \Mod 4.$ Then the following table gives the number of solutions of the congruence equation in the appropriate interval:

\begin{table}[ht]
\centering
\begin{tabular}{c|c|c}
equation & interval & the number of solutions in the interval \\
\hline
$x^2 \equiv q \Mod{2b}$ & $1\leq x \leq 2b$ & $2^{\omega(2b)-1}$ \\
$x^2 \equiv q \Mod{qb}$ & $1\leq x \leq |q|b$ & $2^{\omega(qb)-1}$ \\
$x^2 \equiv q \Mod{2qb}$ & $1\leq x \leq 2|q|b$ & $2^{\omega(2qb)-2}$
\end{tabular}
\end{table}

\newpage Let $q\equiv 5 \Mod 8.$ Then the following table gives the number of solutions of the congruence equation in the appropriate interval:

\begin{table}[ht]
\centering
\begin{tabular}{c|c|c}
equation & interval & the number of solutions in the interval \\
\hline
$x^2 \equiv q \Mod{2b}$ & $1\leq x \leq 2b$ & $2^{\omega(2b)-1}$ \\
$x^2 \equiv q \Mod{4b}$ & $1\leq x \leq 4b$ & $2^{\omega(4b)}$ \\
$x^2 \equiv q \Mod{qb}$ & $1\leq x \leq |q|b$ & $2^{\omega(qb)-1}$\\
$x^2 \equiv q \Mod{2qb}$ & $1\leq x \leq 2|q|b$ & $2^{\omega(2qb)-2}$\\
$x^2 \equiv q \Mod{4qb}$ & $1\leq x \leq 4|q|b$ & $2^{\omega(4qb)-1}$
\end{tabular}
\end{table}

Let $q\equiv 1 \Mod 8$ and $n\in\mathbb{Z}$ such that $n\geq 0$. Then the following table gives the number of solutions of the congruence equation in the appropriate interval:
\begin{table}[ht]
\centering
\begin{tabular}{c|c|c}
equation & interval & the number of solutions in the interval \\
\hline
$x^2 \equiv q \Mod{2b}$ & $1\leq x \leq 2b$ & $2^{\omega(2b)-1}$ \\
$x^2 \equiv q \Mod{4b}$ & $1\leq x \leq 4b$ & $2^{\omega(4b)}$ \\
$x^2 \equiv q \Mod{2^{n+3}b}$ & $1\leq x \leq 2^{n+3}b$ & $2^{\omega(2^{n+3}b)+1}$ \\
$x^2 \equiv q \Mod{qb}$ & $1\leq x \leq |q|b$ & $2^{\omega(qb)-1}$\\
$x^2 \equiv q \Mod{2qb}$ & $1\leq x \leq 2|q|b$ & $2^{\omega(2qb)-2}$\\
$x^2 \equiv q \Mod{4qb}$ & $1\leq x \leq 4|q|b$ & $2^{\omega(4qb)-1}$\\
$x^2 \equiv q \Mod{2^{n+3}qb}$ & $1\leq x \leq 2^{n+3}|q|b$ & $2^{\omega(2^{n+3}qb)}$ \\
\end{tabular}
\end{table}

\end{lemma}

\begin{proof}
We prove the statements $\Mod{qb}$. The number of solutions $\Mod b$ such that $1\leq x \leq b$ is $2^{\omega(b)}.$ Take any such solution $x_0$ and look at the numbers $x_0, x_0+b, x_0+2b,\dots, x_0+(|q|-1)b.$ Exactly one of them will be divisible by $q$ so there will be exactly $2^{\omega(b)}=2^{\omega(qb)-1}$ solutions $\Mod{qb}$ in the interval $[1,|q|b].$

The other cases are analogous, and the cases where the modulus and $q$ are coprime also follow from \cite[V.4.g]{Vinogradov}.
\end{proof}


\subsection{Sets of good primes \texorpdfstring{$\G_q$}{G_q} for various \texorpdfstring{$q$}{q} with accompanying arithmetic functions}\label{sub:BofN}

As before with Lemma \ref{lemma:2-2}, Lemmas \ref{Lemma:2}, \ref{lema:existence51} and \ref{lema:predznak} motivate definitions of sets $\G_q$ for various residues of $q \Mod 8.$ For $q \equiv 3 \Mod 4$ such that $|q|$ is prime, denote by $\mathcal{G}_q$ the set of good primes for $q$ \[\displaystyle \mathcal{G}=\mathcal{G}_q=\lbrace p\in \mathbb{P} \colon \left(\frac{p}{|q|}\right) = (-1)^{\frac{p-1}{2}}\rbrace,
\] 
and in addition for $q=-1$ let

\[\displaystyle \mathcal{G}=\mathcal{G}_{-1}=\lbrace p\in \mathbb{P} \colon p\equiv 1\Mod 4 \rbrace .\]

Let
\[
\lambda_\mathcal{G}(n)=
	\begin{cases}
		1, \quad\text{ if } n=p_1^{\alpha_1}\dots p_k^{\alpha_k}, \quad p_i\in \mathcal{G}\\
		0, \quad\text{ otherwise}
	\end{cases},
\]

\noindent along with\[ b_q(n)=2^{\omega(n)}\cdot \lambda_{\G_q}(n). \]

We want to estimate the weighted sum
\[
\displaystyle B_q(N)= \sum_{1\leq n \leq N} 2^{\omega(n)}\cdot \lambda_{\G}(n)=\sum_{1\leq n \leq N} b_q(n).
\] $B_q(N)$ counts the total number of solutions $x \in \{1, \dotsc, n\}$ of all congruences $x^2 \equiv q \Mod{n},$ where $\gcd(n,2q)=1$ and $1 \leq n\leq N.$ We can easily account for possible factors of $2$ and $q$ in $n$ later; understanding the asymptotic behavior of $B_q(N)$ will be enough to understand $D_{2,q}(N).$

	As before, we define the following two Dirichlet series (which both depend on $q$):
\[
\zeta_{\G}(s):=\DD \lambda_{\G}(s)=\sum \frac{\lambda_{\mathcal{G}}(n)}{n^s}, \quad \beta_q(s):=\DD b_q(s)=\sum \frac{b_q(n)}{n^s},
\]
for which Lemma \ref{lema:omjerzeta} holds. We rewrite $\beta_q(s)$ in terms of the zeta function and the $L$-function of a Dirichlet character mod $4q$, as these functions are holomorphic in the region $\Re s \geqslant 1$, except for $s=1$, and their values and residues are computable.
\begin{lemma}\label{betaL} With notation as above,
\[
\beta_q(s)=\frac{\zzg^2(s)}{\zzg(2s)} = \frac{\zeta(s)}{\zeta(2s)}\cdot  \frac{L(s,\chi_{4|q|,4|q|-1})}{(1+2^{-s})(1+|q|^{-s})}, 
\]
\[
\beta_{-1}(s)=\frac{\zeta(s)}{\zeta(2s)}\cdot  \frac{L(s,\chi_{4,3})}{(1+2^{-s})}. 
\]
\end{lemma}
\begin{proof}
As in the proof of Lemma \ref{lema:zeteZa2} we first rewrite $\zzg(s)$:
\begin{align*}
    \zzg(s) &= \prod_{p\in\G} (1-p^{-s})^{-1} = \zeta(s) \prod_{p\not\in\G} (1-p^{-s}) \\
            &= \zeta(s)\cdot(1-2^{-s})(1-|q|^{-s})\cdot\prod_{\substack{p\notin \G \\ p \neq 2,|q|}}(1-p^{-s}).
\end{align*}
Plugging this in the expression for $\beta_q(s)$ we have
\begin{align*}
    \frac{\zzg^2(s)}{\zzg(2s)} &= \frac{\zeta^2(s)}{\zeta(2s)} \cdot \frac{(1-2^{-s})^2(1-|q|^{-s})^2}{(1-2^{-2s})(1-|q|^{-2s})} \cdot \prod_{\substack{p\notin \G \\ p \neq 2,|q|}}\frac{(1-p^{-s})^2}{(1-p^{-2s})}\\
    &= \frac{\zeta^2(s)}{\zeta(2s)} \cdot \frac{(1-2^{-s})(1-|q|^{-s})}{(1+2^{-s})(1+|q|^{-s})} \cdot \prod_{\substack{p\notin \G \\ p \neq 2,|q|}}\frac{(1-p^{-s})}{(1+p^{-s})}\cdot \prod_{p\in \G}\frac{(1-p^{-s})}{(1-p^{-s})}\\
    &=\frac{\zeta(s)}{\zeta(2s)}\cdot \frac{1}{(1+2^{-s})(1+|q|^{-s})}\cdot \prod_{\substack{p\notin \G \\ p \neq 2,|q|}} (1+p^{-s})^{-1} \prod_{p\in \G} (1-p^{-s})^{-1}\\
    &=\frac{\zeta(s)}{\zeta(2s)}\cdot \frac{L(s,\chi_{4|q|,4|q|-1})}{(1+2^{-s})(1+|q|^{-s}).}
\end{align*}
The statement for $\beta_{-1}$ follows the same proof, except there is no  $(1-|q|^{-s})$ factor in $\zzg$ (and consequently, no $(1+|q|^{-s})^{-1}$ in $\beta_{-1}$).
\end{proof}

\begin{proposition}
If $q\equiv 3 \Mod 4$ such that $|q|$ is prime, then $B_q(N) \sim \frac{2|q|}{3(|q|+1)}\frac{L(1,\chi_{4|q|,4|q|-1})}{\zeta(2)} N ,$ while $B_{-1}(N) \sim \frac{2}{3}\frac{L(1,\chi_{4,3})}{\zeta(2)}N.$ 
\end{proposition}
\begin{proof} Analogous to the proof of Proposition \ref{prop:07}.
\end{proof}

For $q\equiv 1,5 \Mod 8$ we set 
\[
\G=\G_q=\left\lbrace p \in \P \colon p \neq 2, \left(\frac{p}{|q|}\right)=1\right\rbrace
\]
and define $\lambda_{\G}, b_q(n), B_q(N), \zeta_{\G}(s), \beta_q(s)$ as in the case $q\equiv 3\pmod{4}$ (accordingly with respect to the appropriate set $\G$).

\begin{lemma} For $q\equiv 5 \Mod 8$ such that $|q|$ is prime we have
\[
\beta_q(s)=\frac{\zzg^2(s)}{\zzg(2s)} = \frac{\zeta(s)}{\zeta(2s)}\cdot  \frac{L(s,\chi_{|q|,|q|-1})}{(1+|q|^{-s})}, 
\]
while for $q\equiv 1 \Mod 8$ such that $|q|$ is prime we have
\[
\beta_q(s)=\frac{\zzg^2(s)}{\zzg(2s)} = \frac{\zeta(s)}{\zeta(2s)}\cdot  \frac{(1-2^{-s})L(s,\chi_{|q|,|q|-1})}{(1+2^{-s})(1+|q|^{-s})}. 
\]
\end{lemma}
\begin{proof}
The proof is similar to the proof of Lemma \ref{betaL}.
\end{proof}
\begin{proposition} \label{prop:15mod8}
If $q\equiv 5 \Mod 8$ then \[
B_q(N) \sim \frac{|q|}{|q|+1}\frac{L(1,\chi_{|q|,|q|-1})}{\zeta(2)} N,
\]
and if $q\equiv 1 \Mod 8$ then 
\[
B_q(N) \sim \frac{|q|}{3(|q|+1)}\frac{L(1,\chi_{|q|,|q|-1})}{\zeta(2)} N .
\]
\end{proposition}
\subsection{The asymptotics of \texorpdfstring{$D_{2,q}(N)$}{D_{2,q}(N)} for prime \texorpdfstring{$|q|$}{|q|}} \label{sub:fin}

We complete the task of calculating the asymptotics of $D_{2,q}(N)$ where $|q|$ is prime. In one step of the proof we will interchange the limit and the series. To show that we can do this, we appeal to the dominated convergence theorem, in the form of Tannery's theorem \cite{Loya}, which we now state.

\begin{theorem}[Tannery]
    For each positive integer $k$, let $\sum_{m=1}^{n_k} a_m(k)$ be a finite sum such that $n_k \to \infty$ as $k\to \infty$. If for each $m$, $\lim_{k\to\infty} a_m(k)$ exists, and there is a convergent series $\sum_{m=1}^\infty M_m$ of nonnegative real numbers such that $|a_m(k)| \leqslant M_m$ for all $k\in\N$ and $1 \leqslant m \leqslant n_k$, then
    \[ \lim_{k\to\infty} \sum_{m=1}^{n_k} a_m(k) = \sum_{m=1}^{\infty} \lim_{k\to \infty} a_m(k);\]
    that is, both sides are well defined (the limits and sums converge) and are equal.
\end{theorem}

\begin{proof}[Proof of Theorem \ref{tm:sviq}]
The proofs of parts $a)$ and $b)$ are similar to the proof of Theorem \ref{tm:pm2}. Part $c)$ is a bit more involved as $C_q(N)$ is more complicated.

According to Lemma \ref{Lemma:4}, the number of corresponding congruence solutions is
\begin{align*}
C_q(N) =& \,\, B_q(N)+B_q\left(\left\lfloor\frac{N}{2}\right\rfloor\right) + 2B_q\left(\left\lfloor\frac{N}{4}\right\rfloor\right)+ 4\sum_{m \geqslant 0} B_q\left(\left\lfloor\frac{N}{2^{m+3}}\right\rfloor\right)+
\\ +&B_q\left(\left\lfloor\frac{N}{|q|}\right\rfloor\right)+B_q\left(\left\lfloor\frac{N}{2|q|}\right\rfloor\right)+2B_q\left(\left\lfloor\frac{N}{4|q|}\right\rfloor\right)+4\sum_{m \geqslant 0} B_q\left(\left\lfloor\frac{N}{2^{m+3}|q|}\right\rfloor\right).
\end{align*}
First, we notice that both sums over $m$ are finite, since when $m$ is large enough the term $\frac{N}{2^{m+3}}$ is strictly smaller than $1, $ hence the function $B_q$ is constantly equal to zero. 
Since $B_q(N) \sim l\cdot N$ by Proposition \ref{prop:15mod8}, the sequence $\left(\frac{B_q(N)}{N}\right)_{N\in\N}$ has a finite supremum $M_B$. We are trying to prove that the limits
\[ \lim_{N\to\infty}\frac{\sum_{m \geqslant 0} B_q\left(\left\lfloor\frac{N}{2^{m+3}}\right\rfloor\right)}{N} \text{ and } \lim_{N\to\infty}\frac{\sum_{m \geqslant 0} B_q\left(\left\lfloor\frac{N}{2^{m+3}|q|}\right\rfloor\right)}{N} \]
exist (and find their value). To apply Tannery's theorem, the bound of the form $\displaystyle \frac{B_q\left(\left\lfloor\frac{N}{2^{m}}\right\rfloor\right)}{N} \leqslant \frac{M_B}{2^{m}}$ is sufficient as its sum over all positive integers $m$ is finite. To prove this bound holds, let $N=2^ma+b, 0\leq b < 2^m.$ Now, $\displaystyle \frac{B_q\left(\left\lfloor\frac{N}{2^{m}}\right\rfloor\right)}{N} =\frac{B_q(a)}{2^ma+b}=\frac{B_q(a)}{a}\frac{a}{2^ma+b}\leqslant M_B\cdot\frac{1}{2^{m}}.$ Analogous argument holds for the elements of the other series.

\newcommand{\limN}{\lim_{N\to\infty}}


This shows we can interchange the limit and the series. Now observe that \[ \limN \frac{B_q\left(\left\lfloor \frac{N}{2^{m+3}}\right\rfloor\right)}{N} = \limN \frac{B_q\left(\left\lfloor \frac{N}{2^{m+3}}\right\rfloor\right)}{\left\lfloor \frac{N}{2^{m+3}}\right\rfloor} \cdot \frac{\frac{N}{2^{m+3}}-\left\lbrace\frac{N}{2^{m+3}}\right\rbrace}{N} = \frac{l}{2^{m+3}}.\] Hence, by Tannery's theorem and the previous observation
\begin{align*}
    \limN \frac{C_q(N)}{N} =& \quad l + \frac l2 + \frac l2 + 4\cdot \sum_{m\geqslant 0}\limN \frac 1N\cdot B_q\left(\left\lfloor \frac{N}{2^{m+3}}\right\rfloor\right) +\\
    & \quad \, \frac{1}{|q|} \left(l + \frac l2 + \frac l2 \right)+ 4\cdot\sum_{m\geqslant 0}\limN \frac 1N\cdot B_q\left(\left\lfloor \frac{N}{2^{m+3}|q|}\right\rfloor\right) \\
    =& \quad \left(2l + 4\cdot \sum_{m\geqslant 0} \frac{l}{2^{m+3}}\right) + \frac{1}{|q|}\left(2l + 4\cdot \sum_{m\geqslant 0} \frac{l}{2^{m+3}}\right) = 3l\cdot\left(1+\frac{1}{|q|}\right),
\end{align*}
so the statement of Theorem \ref{tm:sviq} follows.

\end{proof}

\begin{remark}
We conjecture that similar results hold for all positive squarefree integers $q$. More precisely, we conjecture that
if $q\equiv 1 \Mod{8}$, then $D_{2,q}(N) \sim \frac{12h(4q)\log(u_{4q})}{\pi^2\sqrt q}N$,
if $q \equiv 5\Mod{8}$, then $D_{2,q}(N) \sim \frac{8h(4q)\log(u_{4q})}{\pi^2\sqrt q}N$,
and $D_{2,q}(N) \sim \frac{6h(4q)\log(u_{4q})}{\pi^2\sqrt q}N$ otherwise (i.~e.\ if $q\not\equiv 1,5\mod{8}$), where $h(n)$ denotes the class number of a (real) quadratic order of discriminant $n$, while $u_n$ denotes the fundamental unit of the same order. By Dirichlet's class number formula, the constants given here are equal to the constants shown in Theorems \ref{tm:pm2} and \ref{tm:sviq}.
\end{remark}

\section{\texorpdfstring{$D(n)$}{D(n)}-triples}
\begin{definition}
Let $a < b < c$. A $D(n)$-triple $\{a, b, c\}$ is called \emph{regular} if $c = a+b+2r$, where $r^2=ab+n$. A $D(n)$-triple $\{a, b, c\}$ is called \emph{irregular} if it is not regular.
\end{definition}

Let $D_{3,n}^\text{reg}(N)$ denote the number of regular $D(n)$-triples $\{a,b,c\}$ such that $a<b<c\leqslant N$.

The following theorem holds for all integers $n$, and its proof is mostly concerned with showing that different cases give at most $O(1)$-triples. We note here that the number of $D(n^2)$-pairs and $D(n^2)$-triples grows faster than a linear function. Namely, a $D(1)$-pair $\{a, b\}$ induces a $D(n^2)$-pair $\{na, nb\}$. Therefore the number of $D(n^2)$-pairs with all elements up to $N$ is greater than or equal to the number of $D(1)$-pairs with all elements up to $\frac{N}{n}$, which grows as $\frac{6}{\pi^2} \frac{N}{n} \log\frac{N}{n}$.
On the other hand, for non-square integers $n$ and primes $p \equiv 1 \Mod{4|n|}$ , the equation $x^2 \equiv n \Mod{p}$ has at least one positive solution $x_0 < \frac p2$ (by quadratic reciprocity). By defining $a=\frac{x_0^2-n}{p} < \frac{p^2/4+p}{p} < p$, we obtain a $D(n)$-pair $\{a, p\}$. Since $a$ is negative only for finitely many cases, in this manner we get at least $\left(\frac{N}{\log N}\right)\cdot\left(\frac{1}{\varphi(4|n|)}\right) - c(n) > N^{1-\epsilon}$ pairs $(a,p)$ which are all different $D(n)$-pairs.
\begin{theorem}[Minor refinement of Theorem \ref{tm:trojkeUvod}]\label{tm:dujeTrojke}
Let $n$ be a non-zero integer. The number of $D(n)$-triples with all elements in the set $\{1, 2, \dotsc, N\}$ is asymptotically equal to the number of regular $D(n)$-triples, which is in turn half the number of $D(n)$-pairs. More precisely, 
\[ D_{3,n} (N) \sim D_{3,n}^\text{reg}(N) \sim \frac{D_{2,n}(N)}{2}. \]
\end{theorem}
\begin{proof}[Proof (also a proof of Theorem \ref{tm:trojkeUvod})]
Since $\{a, b, c\}$ is a $D(n)$-triple, there exist positive integers $r, s, t$ satisfying $ab+n=r^2, ac+n=s^2, bc+n=t^2$.
According to \cite[Lemma 3]{DujeN}, there exist integers $e, x, y, z$ such that
\[ ae+n^2=x^2, be+n^2=y^2, ce+n^2=z^2, \]
and \begin{equation}\label{eq:cbye}
    c = a+b+ \frac{e}{n}+\frac{2}{n^2}(abe+rxy),
\end{equation}
We consider three cases, depending on the sign of $e$.

\begin{enumerate}
  \item[1)] If $e<0$, then $c \leq n^{2}$. Hence, the number of such triples is $O(1)$ (it is less then $\frac{n^{6}}{6}$, so the implied constant in $O$ depends on $n$ ).

  \item[2)] If $e=0$, then $c=a+b+2 r$. Also, $b=a+c-2 s$, where $a c+n=s^{2}$, $s \geq 0$. Every pair $\{a, c\}, a c+n=s^{2}, a<c \leq N$ induces a regular $D(n)$ triple $\{a, a+c-2 s, c\} \subseteq\{1,2, \ldots, N\}$, unless $a+c-2 s > N$, $a+c-2 s \leq 0$, or $a+c-2 s=a$, or $a+c-2 s=c$. The inequality $a+c-2s > N$ implies $a-2s > N - c \geq 0$. However, $a > 2s$ implies $-4n > a(4c-a) > a\cdot 3c$, which can hold only if $c < \frac 43 |n|$. Therefore the contribution of this case is $O(n)=O(1)$.
  
  Before analyzing the remaining degenerate cases, let us note here that $a+c-2s < 0$ is equivalent to $(c-a)^2 < 4n$.  Assume that $a+c-2s=0$. Then $(c-a)^{2}=4 n$. Hence, this case is impossible if $n$ is not a perfect square. If $n$ is a perfect square, then we obtain $c=a+2 \sqrt{n}$, and therefore the contribution of this case is $N+O(1)$.

    The case $a+c-2 s<0$, after squaring gives $(c-a)^2 < 4n$, which is impossible for $n<0$, while for $n>0$ we have $c<a+2 \sqrt{n}$, which implies $(c-\sqrt{n})^{2}<ac+n<(c+\sqrt{n})^{2}$. If we put $a c+n=(c-\alpha)^{2}$, we find that $|\alpha|<\sqrt{n}$ and $c \mid\left(n-\alpha^{2}\right)$. Hence, $c \leq n$, and the contribution of this case is $O(1)$.

If $a+c-2 s=a$, then $c^{2}-4 a c=4 n$, and $c \leq 4|n|$, while if $a+c-2 s=c$, then $1\cdot 3 c<a(4 c-a)=4|n|$. Hence, the contribution of these both cases is $O(1)$.

Note that every regular $D(n)$-triple $\{a, b, c\}$ is obtained twice by this construction: from $\{a, c\}$ and from $\{b, c\}$. Thus, the total contribution of the case 2), i.e the number of regular $D(n)$-triples, is
\[
D_{3,n} = \frac{1}{2}\left(D_{2,n}(N)-N \cdot[n \text { is a square }]+O(1)\right) .
\]
Here we use the convention that if $S$ is any statement which can be true or false, then the bracketed notation $[S]$ stands for 1 if $S$ is true, and 0 otherwise.

  \item[3)] If $e \geq 1$, then

\[ 
c=a+b+\frac{e}{n}+\frac{2 a b e}{n^{2}}+\frac{2 \sqrt{\left(a b+n\right)\left(ae+n^{2}\right)\left(b e+n^{2}\right)}}{n^{2}}>\frac{2 a b}{n^{2}} .
\]

For now, let us assume that $ab > n$. We have $N \geq c \geq \frac{2 a b}{n^{2}}>\frac{r^{2}}{n^{2}}$. Let us estimate the number of such pairs $\{a, b\}$ satisfying
\[
a b+n=r^{2}, \quad r<|n| \sqrt{N} .
\]
Consider the congruence $x^{2} \equiv n(\bmod a)$. In each interval of the size $a$, there are at most $2^{\omega(a)+1}$ solutions. Hence, the number of pairs $\{a, b\}$ is bounded above by

\begin{align*}
\sum_{a=1}^{|n| \sqrt{N}} 2^{\omega(a)+1} \cdot\left(\frac{|n| \sqrt{N}}{a}+1\right) &= 2 |n| \sqrt{N} \sum_{a=1}^{|n| \sqrt{N}} \frac{2^{\omega(a)}}{a}+2 \sum_{a=1}^{|n| \sqrt{N}} 2^{\omega(a)} \\
&=O\left(\sqrt{N} \log ^{2} N\right)+O(\sqrt{N} \log N) \text{ by \cite[9.3.12]{murty}}\\
&=O\left(\sqrt{N} \log ^{2} N\right)
\end{align*}

On the other hand, if $ab\leq n$, adding at most $O(n^2)$-pairs $\{a, b\}$ to the above estimate does not change it.

If $a$ and $b$ are given, then finding $c$ is equivalent to choosing a solution of the Pellian equation
\[
b s^{2}-a t^{2}=n(b-a) .
\]
Each solution belongs to some recursive sequence (growing exponentially). Hence, in each sequence there are $O(\log N)$ solutions with $s \leq N$.

 The number of the sequences is bounded by $2^{k+\omega(n)+1}$, where $k=\omega(b-a)$ (this bound can be found in \cite[p.399]{DujeBook} and in this reference one can also find previously stated results about Pellian equations). We have $b-a \geq p_{1} \cdots p_{k}$ (product of first $k$ primes) and $\log b>\log (b-a) > \frac 12 p_k > \frac{1}{2} k \log k$. The last inequalities follow by \cite{Rosser} and \cite[Theorem 4, Theorem 18]{RosserSchoenfeld} for $p_k > 16$. For products of smallest $k \in \{2, \dotsc, 6\}$ primes, one confirms it directly, while for $k=1$, the intermediate inequality does not hold, but $\log p_1 = \log 2 > \frac{1}{2}\log 1$ holds.
 
 Therefore, we can conclude that
\[
2^{k}<2^{\frac{2 \log b}{\log k}}<b^{\frac{1.4}{\log k}} .
\]
If $2^{k} \geq b^{0.01}$, then we have $k<e^{140}$ and $b<2^{100 \cdot e^{140}}$, hence, the number of such sequences is $O(1)$. If $2^{k}<b^{0.01}$, then the number of the corresponding sequences is less that $2 \cdot 2^{\omega(n)} \cdot N^{0.01}$. Therefore, the contribution of the case $3)$ is
\[
O\left(\sqrt{N} \log ^{2} N \cdot N^{0.01} \cdot \log N\right)=O\left(N^{0.52}\right) .
\]
\end{enumerate}
\end{proof}
\begin{remark}
For $n=1$, we can refine the estimate for the number of irregular triples. Indeed, if $\{a, b, c\}$ is an irregular $D(1)$-triple, then there exists $0<c_{0}<\frac{c}{4 a b}$ such that $\left\{a, b, c_{0}, c\right\}$ is a regular $D(1)$-quadruple $\left(c_{0}=d_{-}\right.$ in the notation of \cite{duje5}). Hence, the number of irregular $D(1)$-triples is bounded by $D_{4}(N)=$ $O(\sqrt[3]{N} \log N)=O\left(N^{0.34}\right)$ (\cite[Theorem 3]{DujeGlavni}).
\end{remark}

Before proceeding, let us record a gap principle for irregular $D(n)$-triples which we have proven as a corollary -- we believe it might be useful for studying $D(n)$-sets.

\begin{lemma}
Let $n$ be a non-zero integer. If an irregular $D(n)$-triple $\{a, b, c\}$ satisfies $a<b<c$ and $c > n^2$, then
\[ c > \frac{3}{n^2}ab.\]
\end{lemma}
\begin{proof} Since $\{a, b, c\}$ is a $D(n)$-triple, there are positive integers $r, s, t$ satisfying $ab+n=r^2, ac+n=s^2, bc+n=t^2$. According to \cite[Lemma 3]{DujeN}, there exists an integer $e$ such that
\[ ae+n^2=x^2, be+n^2=y^2, ce+n^2=z^2, \]
and \begin{equation} \label{eq:cPrekoe}
    c = a+b+ \frac{e}{n}+\frac{2}{n^2}(abe+rxy),
\end{equation}
where $x=at-rs, y=bs-rt$ and $z=cr-st$. Now we look at two cases.

\begin{itemize}
\item For $n > 0$, we show that both $x$ and $y$ are negative.
Namely, $x$ being negative is equivalent to $at < rs$, i.~e.~
$a\sqrt{bc+n} < \sqrt{a^2bc+n(ab+ac)+n^2}$. Dividing by $a$ and squaring gives an equivalent inequality $bc+\frac{n(b+c)}{a}+\left(\frac{n}{a}\right)^2 > bc +n$, which holds since $c>a$. Analogously one shows that $y$ is negative.

We now show that $c>n^2$ implies $e\geqslant 0$. Since $ce+n^2=z^2$ is non-negative, this means that $e \geqslant \frac{-n^2}{c} > -1$. So $e\geqslant 0$ because it's an integer.

Irregularity of our triple implies that $e\neq 0$. This was already noted in \cite{DujeN}, but in the context of quadruples, so we provide the proof. Assuming $e=0$ implies $x=-n, y=-n$, so $c=a+b+2r$, which would imply that $\{a, b, c\}$ is regular, contrary to our assumption.

Therefore, $e\geqslant 1$. Since $r^2x^2y^2=(ab+n)(ae+n^2)(be+n^2)\geqslant a^2b^2$, then equation \eqref{eq:cPrekoe} implies that $c\geqslant 4\frac{ab}{n^2}$ for positive $n$.

\item For $n < 0$, if $x<0$, then $a t<r s$ implies that $a^{2}(b c+n)<(a b+n)(a c+n)$, which yields $a^{2}>(b+c) a+n$, and $r^{2}=a b+n<a(a-c)<0$, a contradiction. Similarly, if $y<0$, then $b s<r t$ implies that $r^{2}=a b+n<b(b-c)<0$, a contradiction. Thus, we have that both $x$ and $y$ are positive.

If $r x y<|n| e$, then $(a b-|n|)\left(a e+n^{2}\right)\left(b e+n^{2}\right)<n^{2} e^{2}$. Since $a b-|n| \geq 1$ by $a b \geq n^{2}$, the LHS of the above inequality is at least $a b e^{2}$, which contradicts $a b \geq n^{2}$. Hence, $r x y \geq|n| e$. It follows that
\begin{align*}
c & =a+b+\frac{2 a b e}{n^{2}}+\frac{2 r x y-|n| e}{n^{2}} \geq a+b+\frac{2 a b e}{n^{2}}+\frac{r x y}{n^{2}}
\end{align*}

Since $e \geq 2$ clearly implies $c>4 a b / n^{2}>3 a b / n^{2}$, it remains to prove that if $e=1$, then $r x y>a b$.

If $a b=2$, then by $n^{2} \leq a b=2$ we have $|n|=1$ and $
r^{2} x^{2} y^{2}=(2-1)(1+1)(2+1)=6>4=a^{2} b^{2}$.

If $a b>2$, then $a b \geq 3$ and
\begin{align*}
r^{2} x^{2} y^{2} & =(a b-|n|)\left(a b+(a+b) n^{2}+n^{4}\right) \\
& >(a b-\sqrt{a b})(a b+2 \sqrt{a b}+1) \\
& =a^{2} b^{2}+\sqrt{a b}(a b-\sqrt{a b}-1)>a^{2} b^{2}
\end{align*}

We thus obtain $r x y>a b$. Therefore, we conclude $c>3 a b / n^{2}$.
\end{itemize}
\end{proof}

Theorem \ref{tm:dujeTrojke}, together with Theorem \ref{tm:pm2} and Theorem \ref{tm:sviq}, immediately gives the following asymptotics for the number of $D(q)$-triples.
\begin{corollary}\label{kor:trojke}
Let $q$ be an integer such that $|q|$ is a prime or $q=-1$. The number of $D(q)$-triples is given by the following.
\begin{itemize}
    \item[a)] For even $q$,
    \[ D_{3,2}(N) \sim \frac{L(1,\chi_{8,5})}{2\zeta(2)}\cdot N, \, \text{ while } 
        D_{3,-2}(N) \sim \frac{L(1,\chi_{8,3})}{2\zeta(2)}\cdot N. \]
    \item[b)] Let $q\equiv 3 \Mod{4}$ such that $|q|$ is prime, or $q=-1$. Then 
 \[ D_{3,q}(N)  \sim \frac{L(1,\chi_{4|q|,4|q|-1})}{2\zeta(2)}\cdot N . \]
    \item [c)] Let $q\equiv 5 \Mod{8}$ such that $|q|$ is prime. Then
\[ D_{3,q}(N)  \sim \frac{L(1,\chi_{|q|,|q|-1})}{\zeta(2)}\cdot N . \]
    \item [d)] Let $q\equiv 1 \Mod{8}$ such that $|q|$ is prime. Then
\[ D_{3,q}(N)  \sim \frac{L(1,\chi_{|q|,|q|-1})}{2\zeta(2)}\cdot N . \]
\end{itemize}
\end{corollary}

\appendix

\section{Arithmetic functions and their Dirichlet series}
To make the paper more self-contained, we collect the basic definitions, notation and results here. Interested readers can find more background in books \cite{apostol} and \cite{MontV}.

\begin{definition}
A \emph{Dirichlet character of modulus $m$} (where $m$ is a positive integer) is a function $\chi \colon \mathbb{Z} \to \mathbb{C}$ which satisfies
\begin{enumerate}
    \item[1)] $\chi(a)\chi(b)=\chi(ab),$
    \item[2)] $\chi(a+m)=\chi(a),$
    \item[3)] $\chi(a) = 0$ if and only if $\gcd(a, m) > 1$
    
  \end{enumerate}
\end{definition}

Our paper uses the following Dirichlet characters:
\begin{itemize}
    \item[1)] $\chi_{8,1},\chi_{8,3}$ and $\chi_{8,5},$ of modulus $8,$ as well as $\chi_{4,3}$ of modulus $4$, are defined by\\
    \begin{center}
\begin{tabular}{ |c|c|c|c|c| } 
 \hline
  & 1 & 3 & 5 & 7 \\ 
  \hline
 $\chi_{8,1}$ & 1 & 1 & 1 & 1\\ 
  \hline
 $\chi_{8,3}$ & 1 & 1 & -1 & -1\\ 
 \hline
 $\chi_{8,5}$ & 1 & -1 & -1 & 1\\
 \hline
 $\chi_{4,3}$ & 1 & -1 & &\\ 
 \hline
\end{tabular}
\end{center}

\item[2)] For any integer $q\equiv 1 \Mod 4$ such that $|q|$ is prime, we denote
\[
    \chi_{|q|,|q|-1}(a) = \left(\frac{q}{a}\right)
\]
\item[3)]  For any integer $q\equiv 3 \Mod 4$ such that $|q|$ is prime, we denote
\[
    \chi_{4|q|,4|q|-1}(a) = \left(\frac{4q}{a}\right),
\]
\end{itemize}
where $\left(\frac{q}{a}\right)$ is the Kronecker symbol.

\begin{definition}
A \emph{Dirichlet $L$-series} is a function of the form
\[
L(s,\chi)=\sum_{n=1}^{\infty} \frac{\chi(n)}{n^s},
\]
where $\chi$ is a Dirichlet character and $s$ is a complex variable with real part greater than one. By analytic continuation, this function can be extended to a meromorphic function on the whole plane and is then called a \emph{Dirichlet $L$-function}, also denoted by $L(s,\chi).$
\end{definition}
Dirichlet had shown that $L(s,\chi)$ is non-zero at $s=1.$ Moreover, the $L$-function is entire whenever $\chi$ is not principal, as is the case for all the Dirichlet characters in our paper which we evaluate at $s=1$.

\begin{lemma}\label{ood1}
Let $\mathcal{G}$ be some set of primes, and $b(n)=2^{\omega(n)}\cdot \lambda_\mathcal{G}(n)$, where $\lambda_\mathcal{G}(n) = 1$ if all prime factors of $n$ are in $\mathcal{G}$, and $0$ otherwise. Then
$|b(n)|=n^{o(1)}$.
\end{lemma}
\begin{proof}
Let $d(n)$ be the number of divisors of $n$. For any $\epsilon > 0$, by \cite[1.3.3]{murty},
\[  2^{\omega(n)} \leqslant d(n) < 2^{(1+\epsilon)\cdot \log(n)/\log(\log(n))} < e^{\log(n)/\log(\log(n))} = n^{1/\log(\log(n))} , \] for sufficiently large $n$. This implies
$2^{\omega(n)} = n^{o(1)}$.
\end{proof}
\begin{proposition}
    Let $f$ be an arithmetic function such that $|f(n)|=n^{o(1)}.$ Then for any $\delta > 0$ the Dirichlet series $\mathcal{D}f$ converges absolutely and uniformly on $\Re{s} \geq 1+\delta,$ and is therefore holomorphic on $\Re{s} > 1.$
\end{proposition}
\begin{proof}
See Theorem 4.5 in \cite{KouKou} and the discussion preceeding it.
\end{proof}

\begin{corollary}\label{kor:holo}
With notation as in Lemma $\ref{ood1}$, the Dirichlet series $\beta$ of $b(n)$ and the Dirichlet series $\zeta_{\mathcal{G}}$ of $\lambda_\mathcal{G}$ are both holomorphic in the region $\Re{s} > 1$.
\end{corollary}

\begin{theorem}[{\cite[Theorem 1.9]{MontV}}]\label{tm:MV}
If $f$ is multiplicative and $\displaystyle \sum_{n=1}^\infty \frac{|f(n)|}{n^{\delta}} < \infty$, where $\delta$ is the real part of $s$, then
\[ \sum_{n=1}^\infty \frac{f(n)}{n^s} = \prod_p \left(1+\frac{f(p)}{p^s}+\frac{f(p^2)}{p^{2s}}+\dots\right).\]
\end{theorem}

\section*{Acknowledgements}
The first three named authors were supported by the Croatian Science Foundation under the projects no.~IP-2018-01-1313 and ~IP-2022-10-5008. The third named author 
was supported by the project "Implementation of cutting-edge research and its
   application as part of the Scientific Center of Excellence for Quantum and Complex Systems, 
   and Representations of Lie Algebras", PK.1.1.02, European Union, European Regional Development Fund. We are grateful to Matija Kazalicki for useful references and numerous discussions about this research. We thank Tomislav ahri Gracin for writing faster code to count the number of pairs and triples. We thank Rudi Mrazović for suggestions which improved the exposition of this paper. We thank Alen Andrašek for pointing us an error in the previous proof of Theorem 3. We thank anonymous referee for numerous comments which have improved the quality of the paper and in particular, for completing the proof of Lemma 24.

\end{document}